\documentclass[12pt,twoside]{amsart}
\usepackage{mathrsfs}
\usepackage{mathtools}
\usepackage{amssymb}
\usepackage{verbatim}
\usepackage{amsmath}
\usepackage{comment}
\usepackage{amsthm,thmtools,xcolor}
\usepackage[colorlinks,linkcolor=blue,citecolor=blue, pdfstartview=FitH]
{hyperref}
\usepackage{bm}
\usepackage{a4wide}
\usepackage[latin1]{inputenc}
\usepackage[T1]{fontenc}
\usepackage{times}
\usepackage{hyperref}
\usepackage{amssymb,latexsym}
\usepackage{enumerate}

\newcommand{\dbar}{\ensuremath{\overline\partial}}

\usepackage{bbm}
\usepackage{etoolbox}
\makeatletter
\patchcmd\maketitle
  {\uppercasenonmath\shorttitle}
  {}
  {}{}
\patchcmd\maketitle
  {\@nx\MakeUppercase{\the\toks@}}
  {\the\toks@}
  {}
  {}{}
\patchcmd\@settitle
  {\uppercasenonmath\@title}
  {}
  {}{}
\patchcmd\@setauthors
  {\MakeUppercase{\authors}}
  {\authors}
  {}{}

\newcommand{\Ric}{\operatorname{Ric}}
\newcommand{\ddbar}{i\partial\bar\partial}
\newcommand{\Scal}{\operatorname{Scal}} 
\newcommand{\inj}{\operatorname{inj}}

\def\w{\wedge}

\makeatletter
\newcommand{\sumprime}{\if@display\sideset{}{'}\sum%
            \else\sum'\fi}
\makeatother

\begin{document}

\numberwithin{equation}{section}

\newtheorem{theorem}{Theorem}[section]
\newtheorem{proposition}[theorem]{Proposition}
\newtheorem{conjecture}[theorem]{Conjecture}
\def\theconjecture{\unskip}
\newtheorem{corollary}[theorem]{Corollary}
\newtheorem{lemma}[theorem]{Lemma}
\newtheorem{observation}[theorem]{Observation}
\newtheorem{definition}{Definition}
\numberwithin{definition}{section} 
\newtheorem{remark}{Remark}
\def\theremark{\unskip}
\newtheorem{kl}{Key Lemma}
\def\thekl{\unskip}
\newtheorem{question}{Question}
\def\thequestion{\unskip}
\newtheorem{example}{Example}
\def\theexample{\unskip}
\newtheorem{problem}{Problem}

\address{Johannes Testorf: Department of Mathematical Sciences, Norwegian University of Science and Technology, Trondheim, Norway}

\email{johannes.testorf@ntnu.no}

\title[Explicit Asymptotics]{Explicit Estimates for the Bergman Kernel Form}

 \author{Johannes Testorf}
\date{\today}

\begin{abstract} Let $(L,e^{-\phi})$ be a positive Hermitian holomorphic line bundle over a compact Riemann surface $X$, and let $\omega=i\partial\overline{\partial}\phi$.  We obtain explicit pointwise estimates for the Bergman form of the tensor power $mL$.  If $\mathrm{Ric}\,\omega\leq\omega$ and the shortest nonconstant closed geodesic has length at least $2\pi$, then
\[
    K_{m\phi}\geq \frac{2m-1}{4\pi}\,\omega,
\]
with sharpness holding for $(\mathbb P^1,\mathcal O_{\mathbb P^1}(2))$.  We also obtain a local version, depending on an upper curvature bound and the injectivity radius, which recovers the first two terms of the Bergman expansion when the curvature is constant. We also find a higher dimensional version.  Under the two-sided bound $-\omega\leq\mathrm{Ric}\,\omega\leq\omega$ and the same closed-geodesic hypothesis, we also prove
\[
K_{m\phi}\leq \frac{m\omega}{2\pi}
 \left(1+\frac{3}{2m}\right).
\]
The lower estimates use the deformation to the tangent space version of the Ohsawa--Takegoshi theorem established by He, Wang, and the author, whereas the upper bound via B\l ocki--Zwonek and isoperimetric inequalities.

\end{abstract}
\thanks{Parts of this work were done during the author's stay at Jagiellonian University in Krak\'ow and Chalmers University in Gothenburg. The author is grateful for the hospitality received there and especially wants to thank Robert Berman, Zbigniew B\l ocki, and David Witt-Nystr\"om for stimulating discussions. The author is also grateful to his advisor Xu Wang for discussions and feedback regarding this article.  }
\maketitle

\tableofcontents

\section{Introduction}
Consider a compact complex manifold $X$ of complex dimension $n$ with a positive line bundle $(L,e^{-\phi})$.
We write
\[
 \omega:=\ddbar\phi.
\]
We denote the canonical bundle of $X$ as $K_X$. Next we define the Bergman form of level $m\in \mathbb N$ as
\begin{align}\label{eq:bergman-definition}
K_{m\phi}(x):=\sup_{F\in H^0(X,K_X+mL)\setminus\{0\}}\frac{i^{n^2}F\wedge\overline{F}e^{-m\phi}(x)}{\int_Xi^{n^2}F\wedge\overline{F}e^{-m\phi}}.
\end{align}
We will give special attention to the case where $X$ is a Riemann surface, i.e., $n=1$. In this case, we use the conventions
\[
 \Ric\omega:=-i\partial\bar\partial\log\omega,
 \qquad
 \Scal_\omega:=\frac{\Ric\omega}{\omega}
\]
for the Ricci and Scalar curvature, respectively, of the K\"ahler metric $\omega$. We also denote
\[
L_0:=\inf\{L(\gamma): \text{$\gamma$ is a nonconstant closed geodesic}\}.
\]
In \cite{RWOT} it was proved that under the conditions $L_0\ge 2\pi$ and $\mathrm{Ric}\,\omega\le \omega$ the following lower bound holds:
\[
K_\phi\ge \frac{\omega}{4\pi}.
\]
Since we have equality in the case $X=\mathbb P^1$ and $\phi = 2\log(1+|z|^2)$, this is sharp.

However, since it is known that
\[
\frac{1}{m}K_{m\phi} \to \frac{\omega}{2\pi}, \qquad \text{for } m\to \infty,
\]
a natural question is whether this bound can be generalized to reflect this asymptotic behavior. The following result extends the estimate to every tensor power and captures the first correction to the leading asymptotic term.
\\~~\\
\textbf{Theorem A.} \emph{Let $X$ be a compact Riemann surface with K\"ahler metric $\omega=i\partial\dbar\phi$ where $\phi$ is the metric of an ample line bundle $(L,e^{-\phi})$. Assume furthermore that $L_0\ge 2\pi$ and $\mathrm{Ric}\,\omega\le \omega$. Then
\begin{align}\label{eq:theorem-A}
K_{m\phi}\ge \frac{2m-1}{4\pi }\omega.
\end{align}
The estimate is sharp for
$X=\mathbb P^1$, $L=\mathcal O_{\mathbb P^1}(2)$, and
$\phi=2\log(1+|z|^2)$.}
\medskip
\noindent

Furthermore, the Bergman kernel form in the case of a polarized Riemann surface enjoys the asymptotic expansion
\begin{align}\label{eq:bergman-expansion}
K_{m\phi}(z)\sim \frac{\omega}{2\pi}\left(m-\frac{1}{2}\mathrm{Scal}_\omega(z)+O(m^{-1})\right).
\end{align}
We refer to
\cite{Tian90,Zelditch1998,Cat99,Lu2000,MaMarinescu2007} for general results about the Bergman-kernel expansion. 

The local statement from which Theorem A follows is more flexible.
\\~~\\
\textbf{Theorem B(1).} \emph{Let $x\in X$ and let $\ell,m\in\mathbb N$.  Suppose that
}\begin{align}{\label{eq:B-assumptions}
 \inj_\omega(x)\geq \pi\sqrt{\frac{\ell}{m}}
 \quad\text{and}\quad
 \Scal_\omega\leq\frac{m}{\ell}
 \quad\text{on }B\!\left(x,\pi\sqrt{\frac{\ell}{m}}\right),
}\end{align}
 where $B(x,r)$ is our notation for the geodesic ball centered at $x$ with radius $r$. {Then
}\begin{align}{\label{eq:theorem-B}
 K_{m\phi}(x)
 \geq
 \frac{2\ell-1}{2\ell}\,\frac{m\omega(x)}{2\pi}.
}\end{align} 

The parameter $\ell$ can be chosen to reflect a prescribed upper curvature bound.
\begin{corollary}[Curvature-dependent asymptotics]\label{cor:curvature-asymptotics}
{Assume that $\Scal_\omega\leq\kappa$ on $X$ for some $\kappa>0$ and that
}\[
 {L_0\geq\frac{2\pi}{\sqrt\kappa}.
}\]
{For all sufficiently large $m$, put $\ell_m=\lfloor m/\kappa\rfloor$.  Then
}\begin{equation}{\label{eq:cor-lower}
 K_{m\phi}(x)
 \geq
 \frac{\omega(x)}{2\pi}
 \left(m-\frac{\kappa}{2}+O_\kappa(m^{-1})\right),
 \qquad x\in X.
}\end{equation}
In particular, if  {$\Scal_\omega\equiv\kappa$, the right-hand side has the same first two coefficients as \eqref{eq:bergman-expansion}. We also stress that the error term is independent of the choice of $x$.
}\end{corollary}
Of course, it is natural to ask if one can also prove an upper bound. We obtain the following theorem.
\\~~\\
\textbf{Theorem C.} \emph{Let $X$ be a compact Riemann surface polarized by $(L,e^{-\phi})$ such that $L_0\ge 2\pi$ and 
\[
-\omega\le \mathrm{Ric\,}\omega\le \omega.
\]
Then we have
\begin{align}\label{eq:theorem-C}
    K_{m\phi}\le\frac{\omega}{2\pi}\left(m+\frac{3}{2}\right).
\end{align}
}
While this clearly gives the correct asymptotic form of the leading term, getting an explicit global estimate which reflects the scalar curvature term is difficult. We also remark that it may be possible to improve this estimate, getting a sharp upper bound also seems to be a very difficult problem.

Our theorem is based on a version of the B\l ocki--Zwonek estimate \cite{BZ15} which was obtained in \cite{Wang24} as well as the isoperimetric inequality. Thus this approach resembles a technique used in \cite{Wang24} to obtain a lower bound. Moreover, what is interesting is that this is an approach to finding upper bounds which does not rely on finding a canonical coordinate system. This is in stark contrast to the approach in previous versions of this article uploaded to ArXiv. In particular we no longer require the use of Eilat's isothermal coordinates \cite{eilat2025} which was a key component previously. 
 
\begin{remark}
    It is known that the condition on $L_0$ cannot be removed entirely in Theorems A and C. Indeed, \cite[Section 3]{Lu01} details a counterexample where $L=K_X$, here we see that control on the curvature alone, (even with control on the genus) does not suffice to find universal bounds.
\end{remark}

\subsection{The case of general \texorpdfstring{$n$}{n}  and an estimate of B\l ocki--Darvas}
Theorem B(1) may be generalized easily to higher dimension. We let $K$ denote the real sectional curvature of $(X,\omega)$.
\\~~\\
\textbf{Theorem B(n).} \emph{Let $X$ be a compact complex manifold of dimension
$n$, polarized by $(L,e^{-\phi})$, and put $\omega=i\partial\bar\partial\phi$.
Fix $x\in X$ and let $\ell,m\in\mathbb N$ satisfy $2\ell>n$.  Suppose that
}\begin{equation}\label{eq:B-injectivity}
 \inj_\omega(x)\geq \pi\sqrt{\frac{\ell}{m}},
\end{equation}
\emph{and we have for any $z\in B(x,\pi\sqrt{\ell/m})$ the sectional curvature bound, 
}\begin{equation}\label{eq:B-curvature-comparison}
 K(U,V)
 \leq \frac{m}{4\ell}, \qquad U \perp V\in T_zX
\end{equation}
\emph{holds.  Then
}\begin{align}{\label{eq:theorem-Bn}
 K_{m\phi}(x)
 &\geq
 \frac{\Gamma(2\ell)}{(4\pi\ell)^n\Gamma(2\ell-n)}
 \frac{(m\omega)^n(x)}{n!}\notag\\
 &=\left(\prod_{j=1}^n\left(1-\frac{j}{2\ell}\right)\right)
 \frac{(m\omega)^n(x)}{(2\pi)^n n!}.
}\end{align}
The proof of Theorem B(n) will be identical to the proof of Theorem B(1). The main point in obtaining this generalization is to find the proper curvature assumption.
\\~~\\
The technique used here also lends itself to an alternative proof of a result by B\l ocki and Darvas in \cite{BD26}. To this end, consider a function $\varphi$ from the class
\[
\mathcal{E_\omega}:=\left\{\varphi\in \mathrm{PSH}(X,\omega): \int_X\omega^n =\int_X(\omega+\ddbar\varphi)^n\right\}.
\]
Assume furthermore that we have the curvature inequality
\begin{align}
    0<a\omega\le \omega+\ddbar \varphi
\end{align}
for some constant $a>0$. Define the Bergman function of level $m$ with respect to $\varphi$ as
\begin{align}
    B^\varphi_m(z):=\sup_{F\in H^0(X,mL)\setminus\{0\}}\frac{|F(z)|^2e^{-m(\varphi(z)+\phi(z))}}{\int_X|F|^2e^{-m(\varphi+\phi)}\omega^n/n!}.
\end{align}
Then we have the following statement from \cite{BD26}.
\begin{proposition}
Let $\varphi$ be as above. Then there exists a sequence $\varepsilon_m=\varepsilon_m(a,\omega)\to 0$ such that
\[
B^\varphi_m\ge \frac{m^n}{(2\pi)^n}(1-\varepsilon_m)a^n.
\]
\end{proposition}
In \cite{BD26}, this was proved using first a local version of the Ohsawa--Takegoshi theorem and then a gluing process using the H\"ormander $\dbar$--estimates. Here we prove the proposition with a global Ohsawa--Takegoshi estimate on $X$, giving an alternate "globalized" proof of the estimate.

\section{Preliminaries}

\subsection{Deformation to the Tangent Space and Ohsawa--Takegoshi Theorem}
Effective lower bounds for Bergman kernels are closely related to $L^2$ extension.  The original Ohsawa--Takegoshi theorem is proved in \cite{OhsawaTakegoshi1987}; a deformation-theoretic formulation adapted to compact manifolds based on an approach by Berndtsson--Lempert \cite{BL} was established in \cite{RWOT}. Other prior developments in this direction include \cite{Blocki13,GZ15}. Stating the theorem itself will require introducing the deformation to the tangent space.

Let $X$ be a compact complex manifold of dimension $n$. Choose a point $x\in X$ and a coordinate system $z:U\to \mathbb C^n$ with $z(x)=0$ and put
\[
\mathcal{B}:=\left\{(\xi,s)\in \mathbb C^n\times\mathbb C: s\xi\in z(U)\right\}.
\]
Now we define a manifold $X_\mathbb C$ by gluing $\mathcal B$ with $X\times \mathbb C^*$ via the change of coordinates
\[
\mu:(\xi,s)\mapsto (z,s)= (s\xi,s).
\]
We call $X_\mathbb C$ the deformation to the tangent space of $x$. We also introduce the shorthand $X_s$ for the fiber $\pi_\mathbb C^{-1}(s)$ of a given $s\in \mathbb C$, where $\pi_\mathbb C:X_\mathbb C\to \mathbb C$ is the projection onto the coordinate $s$.  We may also extend the map $\mu$ via the identity to all of $X_\mathbb C$. 

Importantly, there exists a $\mathbb C^*$-action on $X_\mathbb C$ given by
\begin{align*}
    a\cdot(z,s)&=(z,a^{-1}s),\quad a\in \mathbb C^*, (z,s)\in X\times \mathbb C^*,\\
    a\cdot(\xi,s)&=(a\xi,a^{-1}s),\quad a\in \mathbb C^*, (\xi,s)\in \mathcal{B}.
\end{align*}

Now assume we are given a complex line bundle $L$ on $X$. We define a corresponding line bundle on $X_\mathbb C$ via
\[
L_\mathbb C:=\mu^*(L\times \mathbb C).
\]
We use the shorthand $L_s:=L_\mathbb C|_{X_s}$ and note that $L_1\cong L$ and that $L_0$ is trivial. We can now state the point case of one of the main results of \cite{RWOT} for projective manifolds.

\begin{theorem}[\cite{RWOT} Second Main Theorem]
    Let $X$ be a projective manifold and $L$ a line bundle on $X$. Let $\phi_\mathbb C$ be a metric on $L_\mathbb C$ such that $(L_\mathbb C,e^{-\phi_\mathbb C})$ is pseudoeffective, and $\phi_\mathbb C$ is $S^1$ invariant with respect to the $\mathbb C^*$-action on $X_\mathbb C$. (We will use the shorthand $\phi_s:=\phi_\mathbb C|_{X_s}$.) Then if we have 
    \[
    \int_{X_0}i^{n^2}d\xi\w d\bar \xi e^{-\phi_0}<\infty,
    \]
    there exists a section $F\in H^0(X,K_X+L)$ with $F(x)=e (x)dz$ (where $e$ is a local normal frame of $L$ such that $\phi(z)=0$ and $dz=dz_1\w ...\w dz_n$) and
    \[
    \int_{X}i^{n^2}F\wedge \overline{F}e^{-\phi_1}\le \int_{X_0}i^{n^2}d\xi\w d\bar \xi e^{-\phi_0}.
    \]
    In particular, the Bergman kernel form satisfies
    \[
    K_\phi(x)\ge \frac{\omega^n(x)}{\int_{X_0}\omega^n(x) e^{-\phi_0}}.
    \]
\end{theorem}

\subsection{Comparison Theorems}
One of the key results in obtaining our estimates will be the Hessian comparison theorem from Riemannian geometry. We will use a $i\partial\dbar$-version from \cite{Wang24}.

\begin{theorem}[$\partial\overline{\partial}$-Hessian Comparison Theorem]\label{thm:HessianComparison}
    Let $X_1,X_2$ be K\"ahler manifolds of equal dimension. Let $(\gamma_i)_{i=1,2}:[0,b]\to X_i$ be unit-speed geodesics such that
        \[
        b\le \min_i\{\mathrm{inj}(\gamma_i(0))\}
        \]
        and for all $t\in [0,b]$, and $v_i\perp\dot \gamma_i(t)$ we have for the sectional curvature estimate
        \[
        K(\dot \gamma_1(t),v_1)\le K(\dot \gamma_2(t),v_2).
        \]
        Let $d_i$ be the distance function on $X_i$ from $\gamma_i(0)$. For increasing (resp. decreasing), smooth $f:(0,b)\to\mathbb R$ we have 
        \[
        i\partial\overline\partial (f\circ d_1(\cdot ,\gamma_1(0)))(V_1,V_1)\ge i\partial\overline\partial (f\circ d_2(\cdot ,\gamma_2(0)))(V_2,V_2)
        \]
        (resp. $\le$) where $V_i\in T_{\gamma_i(t)}X_i$ with $|V_1|=|V_2|$, 
        \[
        \langle\dot\gamma_1(t),V_1\rangle= \langle\dot\gamma_2(t),V_2\rangle, \quad  \langle\dot\gamma_1(t),JV_1\rangle= \langle\dot\gamma_2(t),JV_2\rangle.
        \]
\end{theorem}
\begin{proof}
   We refer to \cite[Lemma 1.13, Theorem A]{greene2006function}, \cite{Wang24}.

\end{proof}
We will also need Klingenberg's estimate, a consequence of the Rauch comparison theorem obtained in \cite{Klingenberg1995}.
\begin{theorem}[Klingenberg]\label{thm:klingenberg}
Let  {$(M,g)$ be a compact Riemannian manifold, and let $L_0$ be the length of
its shortest nonconstant closed geodesic.  If the sectional curvature
satisfies $K_g\leq\kappa$ for some $\kappa>0$, then
}\begin{equation}{\label{eq:klingenberg}
 \frac{L_0}{2}\ge\inj_g(M)\geq
 \min\left\{\frac{\pi}{\sqrt\kappa},\frac{L_0}{2}\right\}.
}\end{equation}{If $K_g\leq0$, then $\inj_g(M)=L_0/2$.
}\end{theorem}

\section{The Lower Bound}
\subsection{Proof of Theorem B}

Fix $x\in X$. We will consider $X$ to be imbued with the scaled metric 
\[\omega_m:=m\omega,
 \qquad
 d_m:=d_{\omega_m}=\sqrt m\,d_\omega.
\]
The assumptions \eqref{eq:B-assumptions} become
\[
 \inj_{\omega_m}(x)\geq\pi\sqrt\ell
 \quad\text{and}\quad
 \Scal_{\omega_m}=\frac{1}{m}\Scal_\omega\leq\frac{1}{\ell}
 \quad\text{on }B(x,\pi\sqrt\ell).
\]
Similarly, in the higher dimensional case, (\ref{eq:B-curvature-comparison}) becomes
\[
K_{\omega_m}(U,V)= \frac{1}{m}K_{\omega}(U,V)
 \leq \frac{1}{4\ell}, \qquad U \perp V\in T_zX.
\]
The key now is to find an appropriate choice of $\phi_\mathbb C$ for the deformation to $x$. To this end define the functions
\[
E(z):=\begin{cases}
    \tan^2\left(\frac{d_m(z,x)}{2\sqrt \ell}\right), &d_m(z,x)<\sqrt{\ell}\pi\\
    \infty, & d_m(z,x)\ge \sqrt{\ell}\pi
\end{cases}
\]
on $X$.  Now we set 
\[
\phi_\mathbb C(z,s):=\mu^*_X(m\phi)+\begin{cases}
2\ell\log
\displaystyle
\frac{|s|^2+\mu_X^*E(z)}
{|s|^2(1+\mu_X^*E(z))},
&0<|s|\leq1,\\[9pt]
0,&|s|\geq1,
\end{cases},
\]
where $\mu_X:=\pi_X\circ\mu:X_\mathbb C\longrightarrow X$. We verify that this is a positive metric using the Hessian comparison theorem.

We will compare with the case where $X_2=\mathbb P^n$ and is polarized by the line bundle $\mathcal{O}_{\mathbb P^n}(2\ell)$. This means that the metric is given by $2\ell\phi_{FS}:=2\ell\log(1+|z|^2)$ and the K\"ahler form by $2\ell\omega_{FS}=2\ell i\partial\overline{\partial}\log(1+|z|^2)$ in any affine chart. A direct computation gives for the point $0\in \mathbb P^n$ that the distance is given by the formula
\[
 d_{2\ell \omega_{FS}}(z,0)=2\sqrt \ell\int_0^{|z|}\frac{dx}{1+x^2}=2\sqrt \ell\arctan|z|.
\]
Hence, for $\mathbb P^n$ with this polarization, we have $\mathrm{inj}_{2\ell \omega_{FS}}(0)=\sqrt \ell\pi$, and we know that for the sectional curvature
\[
K(V,W)\ge\begin{cases}\frac{1}{\ell} & n=1\\
        \frac{1}{4\ell}&n>1\end{cases},
\]
with equality in the $n=1$ case.
Now we will generalize the argument from \cite{RWOT,Wang24} to the ${\ell}>1$ case. Since the argument does not differ in any deep way, we keep the exposition brief.
Set
\[
\widehat v_t(z):=2\ell\log\frac{1+e^tE(z)}{1+E(z)}
 \quad(t\geq0),
 \qquad
 \widehat v_t:=0\quad(t\leq0).
\]
For $0\leq\alpha\leq2\ell$, consider the partial Legendre transform
\[
 v_\alpha:=\inf_{t\geq0}(\widehat v_t-\alpha t).
\]
For $0<\alpha<2\ell$, direct minimization gives
\begin{equation}\label{eq:valpha-explicit}
 v_\alpha(z)=
 \begin{cases}
  \alpha\log E(z)-2\ell\log(1+E(z))+c_{\alpha,\ell},
  &0\leq E(z)\leq \dfrac{\alpha}{2\ell-\alpha},\\[7pt]
  0,&E(z)\geq \dfrac{\alpha}{2\ell-\alpha},
 \end{cases}
\end{equation}
where
  \[
 c_{\alpha,\ell}
 =2\ell\log\frac{2\ell}{2\ell-\alpha}
 -\alpha\log\frac{\alpha}{2\ell-\alpha};
\]
Similarly to the proof of \cite[Theorem A]{RWOT} we have by direct computation that this is a bounded test curve on $\mathbb P^n$, (i.e., a concave decreasing map into the psh functions which is identically zero for an $\alpha\le 0$ and $-\infty$ for $\alpha>>0$)  which means that $\widehat v_t$ was actually a subgeodesic ray, in other words $m\phi$-psh. By using the $\partial\dbar$-comparison theorem, we get that this is also a test curve on $X$. Thus, $\phi_\mathbb C$ is an admissible metric for the Ohsawa-Takegoshi theorem on $X$.

We now compute the limit in the central fiber. Choose holomorphic coordinates $w=(w_1,....,w_n)$ centered at $x$ and normalized by
\[
 \omega_m(x)=\frac{i}{2}\sum_{j=1}^n\,dw_j\wedge d\bar w_j,
\]
and choose a local frame of $mL$ for which $m\phi(x)=0$. Then
\[
 d_m(w,x)=|w|+O(|w|^2),
\]
so, in the deformation chart $w=s\xi$,
\[
 E(s\xi) =\frac{|s|^2|\xi|^2}{4\ell}+O(|s\xi |^3).
\]
Hence, we get
\[
\lim_{\substack{s\to 0\\s\neq 0}}\phi_{\mathbb C}(\xi,s)=\lim_{\substack{s\to 0\\s\neq 0}}\left(m\phi(s\xi)+2\ell\log\frac{1+\frac{|\xi|^2}{4\ell}+O(|\xi|^3|s|)}{1+E(s\xi)}\right)=2\ell \log\left(1+\frac{|\xi|^2}{4\ell}\right).
\]
Now we apply the Ohsawa-Takegoshi theorem to find a form $F\in H^0(X,K_X+mL)$ with $F(x)=e(x)\otimes dw$ and 
\begin{align}
    \int_Xi^{n^2}F\wedge\overline{F}e^{-m\phi}\le \int_{\mathbb C^n}\frac{i^{n^2}d\xi\wedge d\bar \xi}{(1+\frac{|\xi|^2}{4\ell})^{2\ell}},
\end{align}
where $d\xi:=d\xi_1\w \cdots\w d\xi_n$.
Hence we have that 
\[
K_{m\phi}(x)\ge 2^{n}\left(\int_{\mathbb C^n}\frac{i^{n^2}d\xi\wedge d\bar \xi}{(1+\frac{|\xi|^2}{4\ell})^{2\ell}}\right)^{-1}\frac{\omega_m^n}{n!}.
\]
We calculate
\begin{align*}
   \int_{\mathbb C^n} \frac{i^{n^2}d\xi\wedge d\bar \xi}{2^n(1+\frac{|\xi|^2}{4\ell})^{2\ell}}&=\frac{2\pi^n}{\Gamma(n)}\int_0^\infty\frac{r^{2n-1}dr}{(1+\frac{r^2}{4\ell})^{2\ell}}\\
    &=\frac{(4\pi \ell)^n}{\Gamma(n)}\int_{0}^\infty\frac{x^{n-1}dx}{(1+x)^{2\ell}}\\
    &=\frac{(4\pi \ell)^n\Gamma(2\ell-n)}{\Gamma(2\ell)}.
\end{align*}
Thus we have proven both Theorem B(1) and Theorem B(n).
\subsection{Proof of Theorem A}
The assumption $\Ric\omega\leq\omega$ is equivalent to $\Scal_\omega\leq1$. By Theorem~\ref{thm:klingenberg} and $L_0\geq2\pi$,
\[
 \inj_\omega(X)\geq\pi.
\]
Apply Theorem B(1) with $\ell=m$. Its hypotheses hold at every point, and \eqref{eq:theorem-A} follows.

To see that this is sharp in the example of $\mathbb P^1$ polarized by $\mathcal{O}(2)$, note first that the Bergman form in this case is a constant multiple of $\omega$. Since the integral of the Bergman form must be equal to 
\[
\dim H^0(\mathbb P^1,K_X+\mathcal{O}(2m)) = 2m-1,
\]
and moreover, $\int\omega=4\pi$, we have that 
\[
K_{m\phi} = \frac{2m-1}{4\pi}\omega.
\]
\subsection{Proof of Corollary \ref{cor:curvature-asymptotics}}
By Klingenberg's estimate,
\[
\mathrm{inj}_\omega(x)\ge \frac{\pi}{\sqrt \kappa}.
\]
Thus for 
\[
\ell_m:=\left\lfloor\frac{m}{\kappa}\right\rfloor,
\]
one has
\[
\pi \sqrt \frac{\ell_m}{m}\le \frac{\pi}{\sqrt \kappa}, \qquad \kappa \le \frac{m}{\ell_m}.
\]
Now Theorem B gives 
\[
K_{m\phi}\ge \frac{\omega}{2\pi}\left(m-\frac{m}{2\ell_m}\right).
\]
But since $m/\ell_m = \kappa+O_\kappa(m^{-1})$, we have
\[
K_{m\phi}\ge \frac{\omega}{2\pi}\left(m-\frac{\kappa}{2}+O_\kappa(m^{-1})\right).
\]

\subsection{Proof of the B\l ocki--Darvas estimate }
The first thing the reader should note is that we want a result about the Bergman kernel of $mL$ rather than $mL+K_X$. To rectify this, we seek a lower bound for the Bergman kernel form of $mL-K_X$. Thus we actually need some bound for the curvatures of this bundle. First, note that we have by compactness that there exists $\rho>0$ such that
\[
 \rho\omega\ge \mathrm{Ric}\,\omega\ge -\rho\omega.
\]
Thus we have for the K\"ahler form $\omega_{\varphi,m}:=m(\omega+\ddbar\varphi)+\mathrm{Ric\,}\omega$ that
\[
\omega_{\varphi,m}\ge (am-\rho)\omega>0.
\]
For convenience, we introduce the notation $a_m:=am-\rho$. By compactness of $X$ we also have that there is a number $\kappa>0$ such that 
\[
K_{a_m\omega}(V,W)\le \frac{\kappa}{m},\quad V\perp W\in T_xX.
\]
Thus, as long as $m$ is chosen large enough, we may find $\ell_m$ with 
\[
4\ell_m\le \frac{m}{\kappa}.
\]
Thus we have for $(\mathbb P^n,\mathcal{O}_{\mathbb P^n}(2\ell_m))$ that
\[
K_{\mathbb P^n}(U_2,V_2) \ge \frac{1}{4\ell_m}\ge K_{a_m\omega}(U_1,V_1), \qquad U_1\perp V_1\in T_xX,\, U_2\perp V_2\in T_0\mathbb P^n. 
\]
We add another constraint in choosing $\ell_m$: it must be such that $\inj_{a_m \omega}(X)\ge \pi\sqrt \ell_m$, which means that the injectivity radius of $X$ is bigger than that of projective space. Thus we may use the comparison theorem to compare $(X,a_m\omega)$ with $(\mathbb P^n,\mathcal{O}_{\mathbb P^n}(2\ell_m))$ and construct the same metric as in the proof of Theorem B. The key now is that the metric is not only psh relative to $a_m\omega$ (or rather $a_m\phi$), but also psh relative to $\omega_{\varphi,m}$ since we have
\[
\omega_{\varphi,m}+\ddbar v_\alpha=(\omega_{\varphi,m}-a_m\omega)+(a_m\omega+\ddbar v_\alpha)\ge 0.
\]
Thus we get the same estimate for the Bergman kernel relative to $\varphi$.
The resulting Ohsawa--Takegoshi estimate reads
\[
B_m^\varphi\ge \frac{a_m^n\Gamma(2\ell_m)}{(4\pi \ell_m)^n\Gamma(2\ell_m-n)} = \left(\frac{(a-\frac{\rho}{m})^n}{a^n}\prod_{j=1}^n\left(1-\frac{j}{2\ell_m}\right)\right)
 \frac{m^na^n}{(2\pi)^n}.
\]
Note that the constraints for $\ell_m$ become unbounded for $m\to \infty$, so we can choose $\ell_m\to \infty$ from which the proposition follows.

\section{Proof of Theorem C}\noindent
Define for $F\in H^0(X,K_X+mL)$,
\[
|F|^2:=\frac{iF\w\bar Fe^{-m\phi}}{\omega}, \qquad \|F\|^2_\infty:=\sup_X|F|^2.
\]
Moreover, set
\[
u:=\log\left(\frac{\|F\|_\infty^2}{|F|^2}\right)\ge 0.
\]
Assume that $F$ is normalized in the $L^2$ norm. Then we have that
\begin{align}\label{eq:thmC1}
    1=\int_{X}iF\w \bar Fe^{-m\phi} = \int_X\|F\|_\infty^2e^{-u}\omega =\|F\|_\infty^2\int_{0}^\infty e^{-t}|\{u<t\}|_\omega dt,
\end{align}
$|\cdot|_\omega$ denotes the measure of a set with respect to the volume form $\omega$. Setting $A(t):=|\{u>t\}|_\omega$, we have by \cite[Lemma 4.1]{Wang24}
\[
A'(t) = \frac{d}{dt}\int_{u<t}\omega\ge\frac{\sigma_\omega(\{u=t\})^2}{2\int_{u<t}\ddbar u}, 
\]
where $\sigma_\omega$ is the measure of the hypersurface $\{u=t\}$ with respect to $\omega$. Now we want apply the isoperimetric inequality \cite[Theorem 5.2]{MJ00} to those $t$ where 
\[
A(t)=|u<t|_\omega\le\frac{1}{2}\int_X\omega.
\]
Now the isoperimetric inequality for Riemann surfaces with $\Scal_\omega\le 1$ gives in this case,
\[
\sigma_\omega(\{u=t\})^2\ge\min\{L_0^2,A(t)(4\pi-A(t))\}=A(t)(4\pi-A(t)).
\]
To handle $t$ with $A(t)>\frac{1}{2}\int_X\omega$ we introduce the function
\[
B(t):=|u> t|_\omega.
\]
Then the isoperimetric inequality gives similarly that
\[
\sigma_\omega(\{u=t\})^2\ge \min\{L_0^2,B(t)(4\pi-B(t))\}.
\]
Moreover, since by comparison with the projective line as before, we have that
\[
\int_X\omega\ge 4\pi,
\]
we have
\begin{align*}
    4\pi &\le A(t)+B(t).
\end{align*}
Now we apply $B(t)\le A(t)$ to get,
\begin{align*}
    4\pi(B(t)-A(t))&\ge B(t)^2-A(t)^2,
\end{align*}
and $B(t)\le 4\pi\le A(t)$ to infer that
\[
B(t)(4\pi-B(t))\ge A(t)(4\pi -A(t)).
\]
Hence, we have also in this case that
\[
\sigma_\omega(\{u=t\})^2\ge\min\{L_0^2,A(t)(4\pi-A(t))\}=A(t)(4\pi-A(t)).
\]
Plugging this into our estimate for $A'$, we get
\[
A'(t)\ge\frac{(4\pi-A(t))}{2(m+1)},
\]
where we have used 
$$\ddbar u=m\omega-\mathrm{Ric\,}\omega\ge (m+1)\omega$$
to evaluate the integral in the denominator. Applying separation of variables to the initial value problem
\[
\tilde A'=\frac{(4\pi-\tilde A(t))}{2(m+1)}, \qquad A(0)=0,
\]
we get that
\[
A(t)\ge4\pi\left(1-e^{-t/[2(m+1)]}\right)
\]

Plugging this into (\ref{eq:thmC1}) we get
\[
\frac{1}{\|F\|_\infty}\ge 4\pi \int_0^\infty\left(1-e^{-t/[2(m+1)]}\right)dt = \frac{4\pi }{2m+3}.
\]
This implies that
\[
K_{m\phi}\le \frac{\omega}{2\pi}\left(m+\frac{3}{2}\right).
\]
\section*{Note on LLM use.}
The author acknowledges the use of ChatGPT 5.6 Sol in the preparation of this article  for improving the writing and presentation of the results. The current version of Theorem C the product of a conversation with ChatGPT 6.0 Astra. Aside from this, all of the mathematical ideas are the author's and all mathematical statements produced by LLM have been verified by the author. 

\bibliographystyle{amsalpha}
\bibliography{ref.bib}

\end{document}